\documentclass{amsart}
\usepackage[margin=1.00in]{geometry}

\usepackage{amsfonts}
\usepackage{eufrak}
\usepackage{setspace}
\usepackage{graphicx}

	\usepackage{mathrsfs}
\usepackage{hyperref}
\usepackage{graphicx}
\usepackage{amsmath, amsthm, amscd, amsfonts, amssymb, graphicx, color}
 \usepackage{url}
\usepackage{enumerate}
 \makeatletter
\let\reftagform@=\tagform@
\def\tagform@#1{\maketag@@@{(\ignorespaces\textcolor{blue}{#1}\unskip\@@italiccorr)}}
\renewcommand{\eqref}[1]{\textup{\reftagform@{\ref{#1}}}}
\makeatother
\usepackage{hyperref}
\hypersetup{colorlinks=true, linkcolor=red, anchorcolor=green,
citecolor=cyan, urlcolor=red, filecolor=magenta, pdftoolbar=true}

\usepackage{epsfig}        
\usepackage{epic,eepic}       

\newtheorem{theorem}{Theorem}
\theoremstyle{plain}

\newtheorem{corollary}{Corollary}

\newtheorem{definition}{Definition}
\newtheorem{example}{Example}

\newtheorem{lemma}{Lemma}

\newtheorem{proposition}{Proposition}
\newtheorem{remark}{Remark}

\numberwithin{equation}{section}

\DeclareMathOperator{\spe}{sp}

\DeclareMathOperator{\diag}{diag}
\def\etal{et al.\,}

  \def\etal{et al.\,}

\begin{document}

 \title[Operator Jensen's  inequality for operator superquadratic ]
{Operator Jensen's  inequality for operator superquadratic functions}
\author[M.W. Alomari]{Mohammad W. Alomari }

\address {Department of Mathematics, Faculty of Science and Information
	Technology, Jadara University,  P.O. Box 733, Irbid, P.C. 21110, Jordan.}
\email{mwomath@gmail.com}

\date{\today}

\begin{abstract}
	In this work, an operator superquadratic function (in operator sense) for positive Hilbert space operators is defined. Several examples with some important properties together with some observations which are related to the operator convexity are pointed out.     Equivalent statements of a non-commutative version of Jensen's inequality for operator superquadratic function are established.   A generalization of the main result to any positive unital linear map is also provided. 
\end{abstract}


\setcounter{page}{1}

\maketitle

\section{Introduction}

Let $\mathscr{B}\left( \mathscr{H}\right) $ be the Banach algebra
of all bounded linear operators defined on a complex Hilbert space
$\left( \mathscr{H};\left\langle \cdot ,\cdot \right\rangle
\right)$  with the identity operator  $1_\mathscr{H}$ in
$\mathscr{B}\left( \mathscr{H}\right) $. A bounded linear operator
$A$ defined on $\mathscr{H}$ is selfadjoint if and only if $
\left\langle {Ax,x} \right\rangle \in \mathbb{R}$ for all $x\in
\mathscr{H}$. The spectrum of an operator $A$ is the set of all
$\lambda \in \mathbb{C}$  for which the operator $\lambda I - A$
does not have a bounded linear operator inverse, and is denoted by
$\spe\left(A\right)$. Consider the real vector space
$\mathscr{B}\left( \mathscr{H}\right)_{sa}$ of self-adjoint
operators on $ \mathscr{H}$ and its positive cone
$\mathscr{B}\left( \mathscr{H}\right)^{+}$ of positive operators
on $\mathscr{H}$. Also, $\mathscr{B}\left(
\mathscr{H}\right)_{sa}^I$ denotes the convex set of bounded
self-adjoint operators on the Hilbert space $\mathscr{H}$ with
spectra in a real interval $I$. A partial order is naturally
equipped on $\mathscr{B}\left( \mathscr{H}\right)_{sa}$ by
defining $A\le B$ if and only if $B-A\in   \mathscr{B}\left(
\mathscr{H}\right)^{+}$.  We write $A
> 0$ to mean that $A$ is a strictly positive operator, or
equivalently, $A \ge 0$ and $A$ is invertible. When $\mathscr{H} =
\mathbb{C}^n$, we identify $\mathscr{B}\left( \mathscr{H}\right)$
with the algebra $\mathfrak{M}_{n\times n}$ of $n$-by-$n$ complex
matrices. Then, $\mathfrak{M}^{+}_{n\times n}$ is just the cone of
$n$-by-$n$ positive semidefinite matrices.

A linear map  is defined to be   $\Phi:\mathscr{B}\left(\mathscr{H} \right)\to \mathscr{B}\left(\mathscr{K} \right)$ which preserves additivity and
homogeneity, i.e.,  $\Phi \left(\lambda_1 A +\lambda_2 B \right)=
\lambda_1\Phi \left( A  \right)+ \lambda_2\Phi \left( B \right)$
for any $\lambda_1,\lambda_2 \in \mathbb{C}$  and $A, B \in \mathscr{B}\left(\mathscr{H} \right)$. The linear map is positive $\Phi:\mathscr{B}\left(\mathscr{H} \right)\to \mathscr{B}\left(\mathscr{K} \right)$ if it preserves the operator order, i.e., if $A\in \mathscr{B}^+\left(\mathscr{H} \right)$ then $\Phi\left(A\right)\in \mathscr{B}^+\left(\mathscr{K} \right)$, and in this case we write
$\mathcal{B} [\mathscr{B}\left(\mathscr{H} \right),\mathscr{B}\left(\mathscr{K} \right)] $. Obviously, a positive linear map $\Phi$ preserves the order relation, namely
$A\le B \Longrightarrow \Phi\left(A\right)\le \Phi\left(B\right)$ and preserves the adjoint operation $\Phi\left(A^*\right)=\Phi\left(A\right)^*$.
Moreover, $\Phi$ is said to be  unital  if it preserves the identity operator,   in this case we write
$\mathcal{B}_n [\mathscr{B}\left(\mathscr{H}
\right),\mathscr{B}\left(\mathscr{K} \right)] $.

A linear map $\Phi:\mathscr{B}\left(\mathscr{H} \right)\to \mathscr{B}\left(\mathscr{K} \right)$  induces  another map 
\begin{align*}
{\rm{id}} \otimes \Phi :\mathbb{C}^{k \times k}  \otimes \mathscr{B}\left( \mathscr{H} \right) \to \mathbb{C}^{k \times k}  \otimes \mathscr{B}\left( \mathscr{K} \right), 
\end{align*}
in a natural way. If $\mathbb{C}^{k \times k}  \otimes \mathscr{B}\left( \mathscr{H} \right)$ is identified with the $C^*$-algebra $\mathscr{B}^{k\times k}\left( \mathscr{H} \right)$ of $k\times k$--matrices  with entries in $\mathscr{B}\left( \mathscr{H} \right)$ then
${\rm{id}} \otimes \Phi $ act as:
\begin{align*}
\left( {\begin{array}{*{20}c}
	{A_{11} } &  \cdots  & {A_{1k} }  \\
	\vdots  &  \ddots  &  \vdots   \\
	{A_{k1} } &  \cdots  & {A_{kk} }  \\
	\end{array}} \right) \mapsto \left( {\begin{array}{*{20}c}
	{\Phi \left( {A_{11} } \right)} &  \cdots  & {\Phi \left( {A_{1k} } \right)}  \\
	\vdots  &  \ddots  &  \vdots   \\
	{\Phi \left( {A_{k1} } \right)} &  \cdots  & {\Phi \left( {A_{kk} } \right)}  \\
	\end{array}} \right).
\end{align*}
We say that $\Phi$  is $k$-positive if ${\rm{id}} \otimes \Phi $  is a positive map, and $\Phi$ is called completely positive if $\Phi$  is $k$-positive for all $k$.

\subsection{Superquadratic functions}
A function $f:J\to \mathbb{R}$ is called convex iff
\begin{align}
f\left( {t\alpha +\left(1-t\right)\beta} \right)\le tf\left(
{\alpha} \right)+ \left(1-t\right) f\left( {\beta}
\right),\label{eq1.1}
\end{align}
for all points $\alpha,\beta \in J$ and all $t\in [0,1]$. If $-f$
is convex then we say that $f$ is concave. Moreover, if $f$ is
both convex and concave, then $f$ is said to be affine.

Geometrically, for two point $\left(x,f\left(x\right)\right)$ and
$\left(y,f\left(y\right)\right)$  on the graph of $f$ are on or
below the chord joining the endpoints  for all $x,y \in I$, $x <
y$. In symbols, we write
\begin{align*}
f\left(t\right)\le   \frac{f\left( y \right)  - f\left( x \right)
}{y-x}   \left( {t-x} \right)+ f\left( x \right)
\end{align*}
for any $x \le t \le y$ and $x,y\in J$.

Equivalently, given a function $f : J\to \mathbb{R}$, we say that
$f$ admits a support line at $x \in J $ if there exists a $\lambda
\in \mathbb{R}$ such that
\begin{align}
f\left( t \right) \ge f\left( x \right) + \lambda \left( {t - x}
\right) \label{eq1.2}
\end{align}
for all $t\in J$.

The set of all such $\lambda$ is called the subdifferential of $f$
at $x$, and it's denoted by $\partial f$. Indeed, the
subdifferential gives us the slopes of the supporting lines for
the graph of $f$. So that if $f$ is convex then $\partial f(x) \ne
\emptyset$ at all interior points of its domain.

From this point of view  Abramovich \etal \cite{SJS} extend the
above idea for what they called superquadratic functions. Namely,
a function $f:[0,\infty)\to \mathbb{R}$ is called superquadratic
provided that for all $x\ge0$ there exists a constant $C_x\in
\mathbb{R}$ such that
\begin{align}
f\left( t \right) \ge f\left( x \right) + C_x \left( {t - x}
\right) + f\left( {\left| {t - x} \right|} \right)\label{eq1.3}
\end{align}
for all $t\ge0$. We say that $f$ is subquadratic if $-f$ is
superquadratic. Thus, for a superquadratic function we require
that $f$ lie above its tangent line plus a translation of $f$
itself. If $f$ is differentiable and satisfies $f(0) = f^{\prime}(0) = 0$, then one sees easily
that the $C_x$ appearing in the definition is necessarily $f^{\prime}(x)$, (see \cite{AJS}).

Prima facie, superquadratic function  looks  to be stronger than
the convex function itself but if $f$ takes negative values then it
maybe considered as a weaker function. Therefore, if $f$ is
superquadratic and non-negative, then $f$ is convex and increasing
\cite{SJS} (see also \cite{A}).

Moreover,   the following result holds for superquadratic
function.

\begin{lemma}\cite{SJS}
	\label{lemma1}Let $f$ be superquadratic function. Then
	\begin{enumerate}
		\item $f\left(0\right)\le 0$
		
		\item If $f$ is differentiable and $f(0)=f^{\prime}(0)=0$, then $C_x=f^{\prime}(x)$  for all $x\ge0$.
		
		\item If $f(x)\ge0$ for all $x\ge 0$, then $f$ is convex and $f(0)=f^{\prime}(0)=0$.
	\end{enumerate}
\end{lemma}

The next result gives a sufficient condition when convexity
(concavity) implies super(sub)quaradicity.

\begin{lemma}\cite{SJS}
	\label{lemma2}If $f^{\prime}$ is convex (concave) and
	$f(0)=f^{\prime}(0)=0$, then is super(sub)quadratic. The converse
	of is not true.
\end{lemma}

\begin{remark}
	In general, non-negative subquadratic functions does not imply concavity. In other words,  there exists a subquadratic function which is convex. For example, $f(x)=x^p$, $x\ge 0$ and $1\le  p \le2$ is subquadratic and convex.
\end{remark}

Among others,  Abramovich \etal \cite{SJS} proved that the
inequality
\begin{align}
f\left( {\int {\varphi d\mu } } \right) \le   \int {f\left(
	{\varphi \left( s \right)} \right)-f\left( {\left| {\varphi \left(
			s \right) - \int {\varphi d\mu } } \right|} \right)d\mu \left( s
	\right)}\label{eq.SJS}
\end{align}
holds for all probability measures $\mu$ and all nonnegative,
$\mu$-integrable functions $\varphi$ if and only if $f$ is
superquadratic.  For more details the reader may refer to
\cite{A}, \cite{SIP}, and \cite{KLPP}.

\subsection{Operator convexity and Jensen inequality }
Let $f$ be a real-valued function defined on $J$. A $k$-th order
divided difference of $f$ at distinct points $x_0,\cdots,x_k$ in $J$
may be defined recursively by
\begin{align*}
\left[ {x_i } \right]f  &= f\left( {x_i } \right) \\
\left[ {x_0 ,x_1 , \ldots ,x_k } \right]f &= \frac{{\left[ {x_1 , \ldots ,x_k } \right]f - \left[ {x_0 , \ldots ,x_{k - 1} } \right]f}}{{x_k  - x_0 }}. 
\end{align*}
For instance, the first $3$-divided differences are given as
follows:
\begin{align*}
\left[ {x_0 } \right]f &= f\left( {x_0 } \right) \,\qquad \qquad \qquad\qquad\qquad {\rm {if}}\qquad k=0,\\
\left[ {x_0 ,x_1 } \right]f &= \frac{{\left[ {x_1 } \right]f -
		\left[ {x_0 } \right]f}}{{x_1  - x_0 }} \,\,\, \qquad
\qquad\qquad\,\, {\rm {if}}\qquad k=1,
\\
\left[ {x_0 ,x_1 ,x_2 } \right]f &= \frac{{\left[ {x_1 ,x_2 } \right]f - \left[ {x_0 ,x_1 } \right]f}}{{x_2  - x_0 }}    \qquad\qquad {\rm {if}}\qquad k=2.
\end{align*}

A function $f: J\to \mathbb{R}$ is said to be matrix monotone of
degree $n$ or $n$-monotone, if for every $A,B\in \mathfrak{M}_{n\times n}$, it is true that $A\le B$
$\Longleftrightarrow$ $f\left(A\right)\le \left(B\right)$.
Similarly, $f$ is said to be operator monotone If $f$ is
$n$-monotone for all $n\in\mathbb{N}$. Also, $f$ is called
operator convex if it is matrix convex ($n$-convex for all $n$);
i.e., if for every pair of selfadjoint operators
$A,B\in \mathfrak{M}_{n\times n}$ we have
\begin{align*}
f\left(\lambda A+ \left(1-\lambda\right)B\right) \le \lambda
f\left(A\right)+\left(1-\lambda\right)f\left(B\right)
\end{align*}
for all $\lambda\in \left[0,1\right]$. If the inequality is
reversed then $f$ is called operator concave. In case we have general Hilbert space $\mathscr{H}$, the above definition holds for every pair of bounded selfadjoint operators $A$ and $B$ in $\mathscr{B}\left(\mathscr{H}\right)$, whose spectra ontianed in $J$.  For more details see
\cite{FK} and the recent survey \cite{PC}.

In 1955, Bendat and Sherman \cite{BS} have shown that $f$ is
operator convex on the open interval $(-1, 1)$ if and only if it
has a (unique) representation
\begin{align*}
f\left( t \right) = \beta _0  + \beta _1 t + \frac{1}{2}\beta _2
\int_{ - 1}^1 {\frac{{t^2 }}{{1 - \alpha t}}d\mu \left( t \right)}
\end{align*}
for $\beta_2\ge0$ and some probability measure  $\mu$ on $[-1, 1]$
(it could be Borel measure). In particular, $f$ must be analytic
with $f(0)=\beta_0$, $f^{\prime}(0)=\beta_1$ and
$f^{\prime\prime}(0)=\beta_2$.\\

We recall that the celebrated L\"{o}wner-Heinz inequality reads that:
\begin{lemma}
	Let $A,B\in \mathscr{B}\left(  \mathscr{H}\right)^+$ such that $A\ge B$, then $A^r\ge B^r$ for all $r\in \left[0,1\right]$.
\end{lemma}
On the other hand the mapping $t \mapsto t^p$ $(p>1)$ is not
operator monotone, for more details see \cite{B}, \cite{FMPS} and
\cite{HP1}.

The classical Jensen's inequality for reals states that
\begin{align}
\label{eq1.5}f\left( {\sum\limits_{j = 1}^n {\lambda _j x_j } }
\right) \le \sum\limits_{j = 1}^n {\lambda _j f\left( {x_j }
	\right)}
\end{align}
valid for all real valued convex function $f$ defined on $[m,M]$,
for every $x_1,\cdots,x_n\in [m,M]$ and every positive real
numbers $\lambda_j$ $(1\le j\le n)$ such that
$\sum_{j=1}^n{\lambda_j}=1$.

The inequality \eqref{eq1.5} would be rephrases under matrix
situation by putting
\begin{align*}
A = \left( {\begin{array}{*{20}c}
	{x_1 } & {} & 0  \\
	{} &  \ddots  & {}  \\
	0 & {} & {x_n }  \\
	\end{array}} \right)\qquad {\rm{and}} \qquad
x = \left( {\begin{array}{*{20}c}
	{\sqrt {\lambda _1 } }  \\
	\vdots   \\
	{\sqrt {\lambda _n } }  \\
	\end{array}} \right)
\end{align*}
then the classical Jensen  inequality \eqref{eq1.5} is expressed as
\begin{align}
f\left( {\left\langle {Ax,x} \right\rangle } \right) \le \left(
{f\left( A \right)x,x} \right)\label{eq1.6}
\end{align}
which is one of the operator version of the classical Jensen's
inequality, see \cite{FMPS}.

Kadison \cite{RK} established his famous
non-commutative version  of the previous inequality where he
proved that for every selfadjoint matrix $A$ the inequality
\begin{align}
\Phi^2\left(A\right)  \le \Phi\left( {A^2} \right)\label{eq1.7}
\end{align}
for every positive unital linear map $\Phi:\mathfrak{M}_{n\times
	n}\left(\mathbb{C}\right)\to \mathfrak{M}_{k\times
	k}\left(\mathbb{C}\right)$.

This inequality was generalized later  by Davis  in \cite{D},
where he obtained that this is true when $f$ is a matrix convex
function and $\Phi$  is completely positive; i.e.,
\begin{align}
f\left(\Phi\left(A\right)\right)  \le \Phi\left( {f\left(A\right)}
\right)\label{eq1.8}
\end{align}
The latter restriction about complete positivity of $\Phi$  was
removed by Choi \cite{C} who proved that \eqref{eq1.7}  remains
valid for all positive unital linear maps   provided $f$ is matrix
convex.

Another noncommutative operator version of the classical Jensen's
inequality under the situation that
\begin{align*}
A = \left( {\begin{array}{*{20}c}
	{x_1 } & {} & 0  \\
	{} &  \ddots  & {}  \\
	0 & {} & {x_n }  \\
	\end{array}} \right)\qquad {\rm{and}} \qquad
V=\left( {\begin{array}{*{20}c}
	{\sqrt{\lambda_1}} & {\cdots} & 0  \\
	{\vdots} &  \ddots  & {}  \\
	\sqrt{\lambda_n} & {} & {0 }  \\
	\end{array}} \right),
\end{align*}
then the classic Jensen's inequality is expressed as
\begin{align}
\label{eq1.9}f\left( {V^* AV} \right) \le V^* f\left( A \right)V.
\end{align}
The inequality \eqref{eq1.9} was proved by Davis in
\cite{D} for all $A\in \mathscr{B}\left(\mathscr{H}\right)$ and
every isometry $C$. However, a more informative version was
extended by Hansen-Pedersen \cite{HP1} as follows:
\begin{theorem}
	\label{theorem1}    Let $\mathscr{H}$ and $\mathscr{K}$ be Hilbert
	space. Let $f$ be
	a real valued continuous function on an interval $I$. Let $A$ and $A_j$ be selfadjoint operators on
	$\mathscr{H}$ with spectra contained in $I $ $(j = 1, 2,\cdots,k)$. Then the following conditions are mutually
	\begin{enumerate}
		\item $f$ is operator convex on $I$ and $f(0)\le0$.
		\item $f\left( {C^* AC} \right) \le C^* f\left( A \right)C$, for every $A\in \mathscr{B}\left(\mathscr{H}\right)$ and contraction $C\in \mathscr{B}\left(\mathscr{H}\right)$; i.e., $C^*C\le1_{\mathscr{K}}$.
		\item $f\left( {\sum\limits_{j = 1}^n {C_j^* A_j C_j } } \right) \le \sum\limits_{j = 1}^n {C_j^* f\left( {A_j } \right)C_j } $, for all $A_j\in \mathscr{B}\left(\mathscr{H}\right)$ and $C_j\in \mathscr{B}\left(\mathscr{H}\right)$ with $\sum_{j=1}^nC_j^*C_j\le 1_{\mathscr{H}}$, $(j = 1, 2,\cdots,k)$.
		\item $f\left( {PAP} \right) \le Pf\left( A \right)P$, for every $A\in \mathscr{B}\left(\mathscr{H}\right)$ and  projection $P$.
	\end{enumerate}
\end{theorem} 

Here we give some popular examples of operator convex and concave
function \cite{PC}.
\begin{enumerate}
	\item For each $p\in \left[0,1\right]$, $t^p$    is operator concave on $\left[0,\infty\right)$.
	
	\item The function $t\log t$  is operator convex on $\left[0,\infty\right)$.
\end{enumerate}

This work is organized as follows: after this introduction; in Section \ref{sec2}, the
operator superquadratic functions for positive Hilbert space operators are
introduced and elaborated. Several examples with some important properties together
with some observations related to operator convexity are pointed
out.  
In Section \ref{sec3}, A Jensen type inequality is proved.
Equivalent statements of a non-commutative version in of Jensen's
inequality for operator superquadratic are also established.   Finally, several trace inequalities for superquadratic functions (in ordinary sense) are provided as well.   
\section{Operator superquadratic function}\label{sec2}

\begin{definition}
	\label{def1}
	Let $I=[0,M]\subseteq \left[0,\infty\right)$.      A real valued continuous function $f (t)$ on an interval $I$ is said to be operator superquadratic function if
	\begin{multline}
	f\left( {\alpha A + \left( {1 - \alpha } \right)B} \right)\\
	\le \alpha \left[ {f\left( A \right) - f\left( {\left( {1 - \alpha } \right)\left| {A - B} \right|} \right)} \right] + \left( {1 - \alpha } \right)\left[ {f\left( B \right) - f\left( {\alpha \left| {A - B} \right|} \right)} \right]\label{eq2.1}
	\end{multline}
	holds for all $\alpha\in \left[0,1\right]$ and for every positive   operators $A$ and $B$ on a Hilbert space $\mathscr{H}$ whose
	spectra are contained in $I\subset [0,\infty)$.  We say that $f$
	is operator subquadratic function if $-f$ is operator superquadratic function. Moreover, if the equality holds in \eqref{eq2.1}, in this case  we say that $f$ is operator quadratic function.
\end{definition}

It's convenient to note that; if $f$ satisfies \eqref{eq2.1}, then
with $A = x$ and $B = y$ (two positive scalars) one can obtain the Jensen inequality for superquadratic functions and if $f$ is continuous (which is necessary to define an operator functions), then \eqref{eq.SJS} would imply that $f$ is superquadratic function. Thus, we observe that:
\begin{corollary}
	\label{cor1}If $f$ is an operator superquadratic function then $f$ is a real superquadratic function.
\end{corollary}

Let $f(t)=\alpha t+\beta$,  then $f$ is  operator subquadratic on
every bounded interval for all $\alpha,\beta \ge 0$. Indeed, we
have
\begin{align*}
&f\left( {\frac{{A + B}}{2}} \right) + f\left( {\frac{{\left| {A - B} \right|}}{2}} \right) - \frac{{f\left( A \right) + f\left( B \right)}}{2}\\
&= \left[ {\alpha \frac{{A + B}}{2} + \beta } \right] + \left[ {\alpha \frac{{\left| {A - B} \right|}}{2} + \beta } \right] - \frac{{\alpha A + \beta  + \alpha B + \beta }}{2} \\
&= \alpha \frac{{\left| {A - B} \right|}}{2} + \beta  \ge 0.
\end{align*}
Moreover,  $g(t)=-f(t)$ is operator superquadratic.

One can easily
seen that the function $t\mapsto t^3$ is not operator
superquadratic nor operator subquadratic function. Simply, assume $f(t)=t^3$,  $t\in [0,\infty)$ and
let
\begin{align*}
A = \left( {\begin{array}{*{20}c}
	2 & 1  \\
	1 & 1  \\
	\end{array}} \right)\qquad\text{and}\qquad B = \left( {\begin{array}{*{20}c}
	1 & 0  \\
	0 & 0  \\
	\end{array}} \right)
\end{align*}
then,
\begin{align*}
\frac{{A^3  + B^3 }}{2} - \left( {\frac{{A + B}}{2}} \right)^3  -
\left( {\frac{{\left| {A - B} \right|}}{2}} \right)^3 &=
\frac{1}{4}\left( {\begin{array}{*{20}c}
	{ 9} & 7  \\
	7 & 5  \\
	\end{array}} \right) \mathop {\nleq }\limits_{\not >}   0.
\end{align*}
However, the map $t\mapsto t^2$ is non-negative operator convex  on
$(0,\infty)$ and it is also  operator superquadratic  on
$(0,\infty)$. Indeed,  by \eqref{eq2.1} we have
\begin{align*}
&\left( {\alpha A + \left( {1 - \alpha } \right)B} \right)^2  \le \alpha A^2  + \left( {1 - \alpha } \right)B^2  - \alpha \left( {1 - \alpha } \right)^2 \left| {A - B} \right|^2  - \left( {1 - \alpha } \right)\alpha ^2 \left| {A - B} \right|^2  \\
&\Leftrightarrow  \alpha^2 A^2 + \left( {1 - \alpha } \right)^2B^2   + \alpha \left( {1 - \alpha } \right)\left( {AB + BA}
\right) \le \alpha A^2  + \left( {1 - \alpha } \right)B^2
\\
&\qquad\qquad\qquad\qquad\qquad\qquad\qquad\qquad\qquad\qquad\qquad- \left[ {\alpha \left( {1 - \alpha } \right)^2  + \left( {1 - \alpha } \right)\alpha ^2 } \right]\left( {A - B} \right)^2  \\
&\Leftrightarrow \alpha \left( {\alpha  - 1} \right)A^2  + \alpha \left( {\alpha  - 1} \right)B^2  + \alpha \left( {\alpha  - 1} \right)\left( {AB + BA} \right) \le \alpha \left( {\alpha  - 1} \right)\left( {A - B} \right)^2  \\
&\Leftrightarrow \alpha \left( {\alpha  - 1} \right)\left( {A +
	B } \right)^2 \le \alpha \left( {\alpha  - 1} \right)\left( {A -
	B} \right)^2   \\
&\Leftrightarrow \left( {A + B} \right)^2  \ge \left( {A - B}
\right)^2\qquad\qquad
{\rm{for  }}\,\,\alpha \left( {\alpha  - 1} \right) < 0\\
&\Leftrightarrow \left| {A + B} \right| \ge \left| {A - B} \right|
\qquad\qquad
g\left( t \right) = \sqrt t \,\,\,{\rm{is\,\, operator \,\,monotone}}    
\end{align*}
which is true  since $A,B>0$, and this proves that $t^2$ is operator superquadratic function.\\

From the definition of operator superquadratic function  we have
\begin{align}
f\left( {\alpha A + \left( {1 - \alpha } \right)B} \right)
\le \alpha \left[ {f\left( A \right) - f\left( {\left( {1 -
			\alpha }
		\right)\left| {A - B} \right|} \right)} \right] + \left( {1 -
	\alpha } \right)\left[ {f\left( B \right) - f\left( {\alpha \left|
		{A - B} \right|} \right)} \right]\label{eq2.2}
\end{align}
for any arbitrary positive  operators $A,B \in
\mathscr{B}\left(\mathscr{H}\right)$ and each $\alpha \in
\left[0,1\right]$.

In particular, by   setting $B= \left\langle {Ax,x} \right\rangle
1_{\mathscr{H}}$ in \eqref{eq2.1} we have
\begin{multline}
f\left( {\alpha A + \left( {1 - \alpha } \right)\left\langle
	{Ax,x} \right\rangle 1_{\mathscr{H}}} \right)
\le \alpha \left[ {f\left( A \right) - f\left( {\left( {1 - \alpha
		} \right)\left| {A - \left\langle {Ax,x} \right\rangle
			1_{\mathscr{H}}} \right|} \right)} \right] \\+ \left( {1 - \alpha }
\right)\left[ {f\left( \left\langle {Ax,x} \right\rangle   \right)
	- f\left( {\alpha \left| {A - \left\langle {Ax,x} \right\rangle
			1_{\mathscr{H}}} \right|} \right)} \right].\label{eq2.3}
\end{multline}
for each positive   operator $A \in
\mathscr{B}\left(\mathscr{H}\right)$ and all $\alpha \in
\left[0,1\right]$.  

From this point of view \eqref{eq2.3}, Kian early in \cite{K} and
then jointly  with Dragomir in \cite{KS}  proved   a finite
dimensional operator version of Jensen's inequality for
superquadratic functions (in ordinary sense) under the
interpretation that for $ A = \left( {\begin{array}{*{20}c}
	a & 0  \\
	0 & b  \\
	\end{array}} \right)$ and $x = \left( {\begin{array}{*{20}c}
	{\sqrt \lambda  }  \\
	{\sqrt {1 - \lambda } }  \\
	\end{array}} \right)$, then we have $ \left\langle {Ax,x} \right\rangle  = \lambda a + \left( {1 - \lambda } \right)b$
if follows that
\begin{align*}
\left| {A - \left\langle {Ax,x} \right\rangle } \right| = \left(
{\begin{array}{*{20}c}
	{\left( {1 - \lambda } \right)\left| {a - b} \right|} & 0  \\
	0 & {\lambda \left| {a - b} \right|}  \\
	\end{array}} \right).
\end{align*}
Therefore, as a matrix Jensen inequality for a superquadratic
function $f:\left[0 ,\infty\right)\to \mathbb{R}$ we have
\begin{align*}
f\left( {\left\langle {Ax,x} \right\rangle } \right) \le
\left\langle {f\left( A \right)x,x} \right\rangle  + \left\langle
{f\left( {\left| {A - \left\langle {Ax,x} \right\rangle } \right|}
	\right)x,x} \right\rangle.
\end{align*}
This result was generalized for positive unital linear maps, as follows:
\begin{theorem}\emph{(\cite{KS}, \cite{MA})}
	\label{thm3}Let $A\in \mathscr{B}\left(\mathscr{H} \right)$ be a
	positive  operator and  $\Phi:\mathscr{B}\left(\mathscr{H}
	\right)\to \mathscr{B}\left(\mathscr{K} \right)$  be a positive unital
	linear map. If   $f:\left[0,\infty\right)\to \mathbb{R}$
	is super(sub)quadratic function, then we have
	\begin{align*}
	\left\langle {\Phi \left( {f\left( A \right)}
		\right)x,x} \right\rangle \ge (\le)  f\left( {\left\langle {\Phi
			\left( {A} \right)x,x} \right\rangle } \right)   + \left\langle
	{\Phi \left( {f\left( {\left| {A - \left\langle
					{\Phi\left(A\right)x,x} \right\rangle 1_{\mathscr{H}} } \right|}
			\right)} \right)x,x} \right\rangle
	\end{align*}
	for every $x\in \mathscr{K}$ with $\|x\|=1$.
\end{theorem}
The above inequality and other consequences were proved later by the
first author of this paper in \cite{MA} where different approach
is used.

\begin{proposition}
	\label{prp1}    Let $f$ be an operator superquadratic function on $I$.
	Then
	\begin{enumerate}
		\item $f\left(0\right)\le 0$.
		
		\item If $f$ is non-negative, then $f$ is operator convex   and $f(0) =0$.
	\end{enumerate}
\end{proposition}

\begin{proof}
	\begin{enumerate}
		\item Setting $A=B=0$ in \eqref{eq2.3} we ge that $f\left(0\right)\le 0$.
		
		\item Since $f$ is continuous and non-negative, then from \eqref{eq2.3} we have
		\begin{align*}
		f\left( {\alpha A + \left( {1 - \alpha } \right)B} \right) &\le \alpha \left[ {f\left( A \right) - f\left( {\left( {1 - \alpha } \right)\left| {A - B} \right|} \right)} \right] + \left( {1 - \alpha } \right)\left[ {f\left( B \right) - f\left( {\alpha \left| {A - B} \right|} \right)} \right]
		\\
		&\le   \alpha f\left( A \right)   + \left( {1 - \alpha } \right)f\left( B \right)
		\end{align*}
		which means that $f$ is operator convex. To show that  $f(0) =0$, we have by part (1) $f(0)\le 0$ and by assumption $f(x)$ is non-negative i.e., $f(x)\ge0$ for all $x\in I$. In particular, $f(0)\ge0$. Thus, $f(0)=0$.
	\end{enumerate}
\end{proof}

\begin{example}
	Let $f(t)=t^{-1}$, then $f$ is non-negative operator convex  on $(0,\infty)$. However, $f$ is not operator superquadratic function on $(0,\infty)$. For instance, let
	\begin{align*}
	A = \left( {\begin{array}{*{20}c}
		3 & 0  \\
		0 & 1  \\
		\end{array}} \right)\qquad \text{and}\qquad B = \left( {\begin{array}{*{20}c}
		1 & 0  \\
		0 & 2  \\
		\end{array}} \right)
	\end{align*}
	Applying  \eqref{eq2.4} for $f(t)=t^{-1}$, we get
	\begin{align*}
	\frac{{A^{ - 1}  + B^{ - 1} }}{2} - \left( {\frac{{A + B}}{2}} \right)^{ - 1}  - \left( {\frac{{\left| {A - B} \right|}}{2}} \right)^{ - 1}
	&= \frac{1}{{12}}\left( {\begin{array}{*{20}c}
		8 & 0  \\
		0 & 9  \\
		\end{array}} \right) - \frac{2}{{12}}\left( {\begin{array}{*{20}c}
		3 & 0  \\
		0 & 4  \\
		\end{array}} \right) - \frac{6}{{12}}\left( {\begin{array}{*{20}c}
		1 & 0  \\
		0 & 2  \\
		\end{array}} \right) \\
	&= \frac{1}{{12}}\left( {\begin{array}{*{20}c}
		{ - 4} & 0  \\
		0 & -11 \\
		\end{array}} \right)   <0
	\end{align*}

\end{example}

\begin{proposition}
	\label{prp2}Let $f$ be a real valued continuous function defined on an interval $\left[0,\infty\right)$. If $f$ is operator  convex  and non-positive then $f$ is operator superquadratic function.
\end{proposition}
\begin{proof}
	Since $f$ is operator convex, then
	\begin{align*}
	\frac{{f\left( A \right) + f\left( B \right)}}{2}  -f\left( {\frac{{A + B}}{2}} \right) \ge0.
	\end{align*}
	But also $f$ is non-positive, so that
	\begin{align*}
	\frac{{f\left( A \right) + f\left( B \right)}}{2}  -    f\left( {\frac{{A + B}}{2}} \right)- f\left( {\frac{{\left| {A - B} \right|}}{2}} \right) \ge   - f\left( {\frac{{\left| {A - B} \right|}}{2}} \right) \ge0
	\end{align*}
	which means that $f$ is operator superquadratic function.
\end{proof}

\begin{example}
	Let $f\left(t\right)=t\log\left(t\right)$, $t\in \left[0,\infty\right)$ it well known that $f$ operator convex. Clearly, $f$ is negative for all $t\in \left(0,1\right) \subseteq \left[0,\infty\right)$. Therefore,    $f\left(t\right)=t\log\left(t\right)$ is  operator superquadratic function for all $t\in \left(0,1\right)$.
\end{example}

\begin{proposition}
	\label{prp3}Let $f$ be a real valued continuous function defined on an interval $\left[0,\infty\right)$. If $f$ is operator  concave  and non-negative then $f$ is operator subquadratic.
\end{proposition}

\begin{proof}
	Since $f$ is operator concave, then
	\begin{align*}
	f\left( {\frac{{A + B}}{2}} \right)- \frac{{f\left( A \right) + f\left( B \right)}}{2}   \ge0.
	\end{align*}
	But also $f$ is non-negative, so that
	\begin{align*}
	f\left( {\frac{{A + B}}{2}} \right) - \frac{{f\left( A \right) + f\left( B \right)}}{2}  + f\left( {\frac{{\left| {A - B} \right|}}{2}} \right) \ge    f\left( {\frac{{\left| {A - B} \right|}}{2}} \right) \ge0
	\end{align*}
	which means that $f$ is operator subquadratic.
\end{proof}

\begin{example}
	Let $f: \left(0,\infty\right)\to \left(0,\infty\right)$, given by
	$f(t)=t^{r}$, $r\in \left[0,1\right]$. Then $f$ is   operator
	subquadratic  on $(0,\infty)$. But $f$ is also operator concave,
	so that
	\begin{align*}
	\frac{{A^r  + B^r }}{2} \le  \left( {\frac{{A + B}}{2}} \right)^r
	\le\left( {\frac{{A + B}}{2}} \right)^r  + \left( {\frac{{\left|
				{A - B} \right|}}{2}} \right)^r
	\end{align*}
	which means $f$ is operator subquadratic  on $(0,\infty)$.
\end{example}

\section{Operator Jensen's inequality}\label{sec3}
In order to prove our results we need the following Lemmas:

\begin{lemma}\emph{(\cite{FMPS})}
	\label{lemma4}If $A \in \mathscr{B}\left(\mathscr{H}\right)$ is selfadjoint and $U$ is unitary, i.e. $U^*U = UU^* = 1_{\mathscr{H}}$, then
	$f (U^*AU) = U^* f (A)U$ for every $f$ continuous on the $\spe(A)$.
\end{lemma}

\begin{lemma}\emph{(\cite{HP})}
	\label{lemma5}Define a unitary matrix $E_n=\diag\left(\xi,\xi^2,\cdots,\xi^{n-1},1\right)$ in $\mathfrak{M}_n\left(\mathbb{C}\right)\subset \mathscr{B}\left(\mathscr{H}^n\right)$, where $\xi =\exp\left(\frac{2\pi i}{n}\right)$. Then for each element $A=\left( a_{ij}\right) \in  \mathscr{B}\left(\mathscr{H}^n\right)$ we have
	\begin{align*}
	\frac{1}{n}\sum_{k=1}^n{E^{-k}AE^k}=\diag\left(a_{11},a_{22},\cdots,a_{nn}\right).
	\end{align*}
\end{lemma}

\begin{lemma}\emph{(\cite{HP})}
	\label{lemma6}Let $P$ denote the projection in $\mathfrak{M}_n$ given by $P_{ij} = n^{-1}$ for
	all $i$ and $j$, so that $P$ is the projection of rank one on the subspace spanned by the vector $\xi+\xi^2+\cdots+\xi^{n}$ in $\mathbb{C}^n$, where $\xi,\xi^2, \cdots,\xi^{n}$ are the standard basis vectors. Then with
	$E$ as in Lemma \ref{lemma5} we obtain the pairwise orthogonal projections $P_k = E^{-k}PE^k$, for $1\le k \le n$, with $\sum_{k=1}^{n}{P_k}=1$.
\end{lemma}

To establish our main first result we need the following
primary result.
\begin{lemma}
	\label{lemma7}    Let $ w_1 , \ldots ,w_n $ be positive real numbers  such that $W_n=\sum_{k=1}^n{w_k}$ and let $A_1,\cdots,A_n$ be positive  operators of a Hilbert space $\mathscr{B}\left(\mathscr{H} \right)$ with spectra contained in a real interval $I$. If $f$ is operator superquadratic function  on $I$, then
	\begin{align}
	f\left( {\frac{1}{{W_n }}\sum\limits_{k = 1}^n {w_k A_k } }\right) \le \sum\limits_{k = 1}^n { \frac{{w_k }}{{W_n }} f\left( {A_k } \right)}  - \sum\limits_{k = 1}^n { \frac{{w_k }}{{W_n }} f\left( {\left| {A_k  -\sum\limits_{j = 1}^n {\frac{{w_j }}{{W_n }}A_j } } \right|}\right)},\label{eq3.1}
	\end{align}
	In particular useful case,  for $w_k=1$ for all $1\le k \le n$, we have
	\begin{align}
	f\left( {\frac{1}{n}\sum\limits_{k = 1}^n {A_k } }\right) \le \frac{1}{n}\sum\limits_{k = 1}^n {  f\left( {A_k } \right)}  - \frac{1}{n}\sum\limits_{k = 1}^n {   f\left( {\left| {A_k  -\frac{1}{n}\sum\limits_{j = 1}^n {A_j } } \right|}\right)},\label{eq3.2}
	\end{align}
\end{lemma}
\begin{proof}
	Assume $f$ is operator superquadratic. If $n=2$, then the inequality
	\eqref{eq3.1} reduces to \eqref{eq2.1} with $\alpha=
	\frac{w_1}{W_2}$ and $1-\alpha= \frac{w_2}{W_2}$. Let us suppose
	that inequality \eqref{eq3.1} holds for $n - 1$. Then for $n$-tuples
	$\left(A_1,\cdots, A_n\right)$ and $\left(w_1,\cdots, w_n\right)$,
	we have
	\begin{align*}
	f\left( {\frac{1}{{W_n }}\sum\limits_{k = 1}^n {w_k A_k } }\right) &= f\left( {\frac{{w_n }}{{W_n }}A_n  + \sum\limits_{k =1}^{n - 1} {\frac{{w_k }}{{W_n }}A_k } } \right)
	\\
	&= f\left( {\frac{{w_n }}{{W_n }}A_n  + \frac{{W_{n - 1} }}{{W_n}}\sum\limits_{k = 1}^{n - 1} {\frac{{w_k }}{{W_{n - 1} }}A_k } }
	\right)
	\\
	&\le  \frac{{w_n }}{{W_n }}\left[f\left( {A_n }
	\right) -f\left( {\frac{{W_{n - 1} }}{{W_n }}\left| {A_n  -
			\frac{1}{{W_{n-1} }}\sum\limits_{k = 1}^{n - 1} {w_k A_k } }
		\right|} \right)\right]
	\\
	&\qquad+  \frac{{W_{n - 1} }}{{W_n }}   \left[
	{f\left( {\sum\limits_{k = 1}^{n - 1} {\frac{{w_k }}{{W_{n - 1}
				}}A_k } } \right) -f\left( {\frac{{w_n }}{{W_n}}\left| {A_n  -
				\frac{1}{{W_{n-1} }}\sum\limits_{k = 1}^{n - 1} {w_k A_k } }
			\right|} \right)} \right]
	\\
	&= \frac{{w_n }}{{W_n }}  f\left( {A_n } \right) +
	\frac{{W_{n - 1} }}{{W_n }}     f\left(
	{\sum\limits_{k = 1}^{n - 1} {\frac{{w_k }}{{W_{n - 1} }}A_k } }
	\right)
	\\
	&\qquad- \frac{{w_n }}{{W_n }}  f\left(
	{\frac{{W_{n - 1} }}{{W_n }}\left| {A_n  - \frac{1}{{W_{n-1}
			}}\sum\limits_{k = 1}^{n - 1} {w_k A_k } } \right|} \right)
	\\
	&\qquad\qquad -  \frac{{W_{n - 1} }}{{W_n }} f\left( {\frac{{w_n }}{{W_n}}\left| {A_n  -
			\frac{1}{{W_{n-1} }}\sum\limits_{k = 1}^{n - 1} {w_k A_k } }
		\right|} \right),
	\end{align*}
	and this is exactly equivalent to write, for any $1\le m\le n$
	\begin{align*}
	f\left( {\frac{1}{{W_m }}\sum\limits_{k = 1}^m {w_k A_k } }\right) \le \sum\limits_{k = 1}^m { \frac{{w_k }}{{W_m }} f\left( {A_k } \right)}
	-\sum\limits_{k = 1}^m { \frac{{w_k }}{{W_m }} f\left( {\left| {A_k  -\sum\limits_{j = 1}^m {\frac{{w_j }}{{W_m }}A_j } } \right|}\right)},
	\end{align*}
	which proves the desired result in \eqref{eq3.1}. The particular case follows by setting $w_k=1$ for all $k=1,\cdots,n$ so that $W_n=n$.
\end{proof}

\begin{remark}
	The result in Lemma \ref{lemma7} was proved by Mond \& Pe\v{c}ari\'{c} in \cite{MP} for all operator convex functions and all bounded selfdjoint operators whose spectra contained in $I$. Therefore, in case $f$ is positive the inequality \eqref{eq3.1} might be considered as a respective extension and new refinement of that result proved in \cite{MP}.
\end{remark}

\begin{theorem}
	\label{thm4}Let $f:I\to \mathbb{R}$ be a real-valued continuous function. Let $\left(A_1,\cdots,A_n\right)$ be an $n$-tuple of positive   of a Hilbert space $\mathscr{H}$ with spectra contained in $I$. Then the following conditions are equivalent:
	\begin{enumerate}
		\item $f$ is operator superquadratic function.\\
		
		\item The inequality
		\begin{align}
		f\left( {\sum\limits_{k = 1}^n {C_k^* A_k C_k } } \right) \le
		\sum\limits_{k = 1}^n {C_k^* f\left( {A_k } \right)C_k }  -
		\sum\limits_{k = 1}^n {C_k^* f\left( {\left| {A_k   -
					\sum\limits_{j = 1}^n {C_j^* A_jC_j  }   }\right|} \right)C_k }\label{eq3.3}
		\end{align}
		holds for every $n$-tuple $\left(C_1,\cdots,C_n\right)$ of
		operators on     $\mathscr{H}$ that satisfy the condition
		$\sum\limits_{k = 1}^n
		{C_k^* C_k }  = 1$. \\

		\item The inequality
		\begin{align}
		f\left( {\sum\limits_{k = 1}^n {P_k A_k P_k } } \right) \le
		\sum\limits_{k = 1}^n {P_k  f\left( {A_k } \right)P_k }  -
		\sum\limits_{k = 1}^n {P_k  f\left( {\left| { A_k   -
					\sum\limits_{j = 1}^n {P_j A_jP_j  }   }
				\right|} \right)P_k }\label{eq3.4}
		\end{align}
		holds   for every $n$-tuple  $\left(P_1,\cdots,P_n\right)$ of
		projections
		on $\mathscr{H}$ with $\sum\limits_{k = 1}^n {P_k }  = 1$.
	\end{enumerate}
\end{theorem}
\begin{proof}
	$(1)\Rightarrow (2)$.  
	We say that $C =\left(C_1, \cdots, C_n\right)$ is a unitary column if there is a unitary $n \times n$ operator matrix $U = (u_{ij})$,
	one of whose columns is $(C_1, \cdots, C_n)$. Thus, $u_{ij} = C_i$ for some $j$ and all $i$.  Assume that we are given a unitary $n$-column $(C_1, \cdots, C_n)$,
	and choose a unitary $U_n = (u_{ij})$ in
	$\mathscr{B}\left(\mathscr{H}^n\right)$ such that $u_{kn} = C_k$.
	Let $E =\diag(\xi, \xi^2, \cdots,\xi^{n-1}, 1)$ as in Lemma \ref{lemma4}
	and put $X =\diag\left(A_1, \cdots,A_n\right)$, both regarded as
	elements in $\mathscr{B}\left(\mathscr{H}^n\right)$. Thus, using   the spectral decomposition theorem,  we have
	\begin{align*}
	f\left( {\sum\limits_{k = 1}^n {C_k^* A_k C_k } } \right)  =
	f\left( {\left( {U_n^* XU_n } \right)_{nn} } \right) = f\left(
	{\left( {\frac{1}{n}\sum\limits_{k = 1}^n {E^{ - k} U_k^* XU_k E^k
		} } \right)_{nn} } \right).
	\end{align*}
	We note that since  $f\left(\diag\left(y_1,
	\cdots,y_n\right)\right)=\diag\left(f\left(y_1\right),
	\cdots,f\left(y_n\right)\right)$, then
	\begin{align*}
	f\left(y_n\right)=f\left(\diag\left(y_1,
	\cdots,y_n\right)\right)_{nn}.
	\end{align*}
	Using the above facts taking into account Lemmas
	\ref{lemma4}--\ref{lemma7} together with the inequality \eqref{eq3.2}, thus  the operator superquadraticity of $f$, implies that
	\begin{align*}
	f\left( {\sum\limits_{k = 1}^n {C_k^* A_k C_k } } \right) &= f\left( {\left( {\frac{1}{n}\sum\limits_{k = 1}^n {E^{ - k} U_k^* XU_k E^k } } \right)_{nn} } \right) \\
	&= \left( {f\left( {\frac{1}{n}\sum\limits_{k = 1}^n {E^{ - k} U_k^* XU_k E^k } } \right)} \right)_{nn}  \\
	&\le \left( {\frac{1}{n}\sum\limits_{k = 1}^n {f\left( {E^{ - k} U_k^* XU_k E^k } \right)} } \right)_{nn}  \\
	&\qquad- \left( {\frac{1}{n}\sum\limits_{k = 1}^n {f\left( { \left| { E^{ - k} U_k^* XU_k E^k-  \frac{1}{n}\sum\limits_{j = 1}^n {E^{ - j} U_j^* XU_j E^j}   } \right|  } \right)} } \right)_{nn}  
\\
	&= \left( {\frac{1}{n}\sum\limits_{k = 1}^n {E^{ - k} U_k^* f\left( X \right)U_k E^k } } \right)_{nn} \\
	&\qquad - \left( {\frac{1}{n}\sum\limits_{k = 1}^n {  f\left( {E^{ - k} U_k^*\left| {X -   \frac{1}{n}\sum\limits_{j = 1}^n {E^{ - j} U_j^* XU_j E^j} } \right|U_k E^k} \right)  } } \right)_{nn}
	\\
	&= \left( {\frac{1}{n}\sum\limits_{k = 1}^n {E^{ - k} U_k^* f\left( X \right)U_k E^k } } \right)_{nn} \\
	&\qquad - \left( {\frac{1}{n}\sum\limits_{k = 1}^n {  E^{ - k} U_k^*f\left( {\left| {X -   \frac{1}{n}\sum\limits_{j = 1}^n {E^{ - j} U_j^* XU_j E^j} } \right|} \right)  } U_k E^k} \right)_{nn}\\
	&= \left( {U_n^* f\left( X \right)U_n } \right)_{nn}  - \left( {U_n^* f\left( {\left| {   X    -  U_n^* XU_n  } \right|} \right)U_n } \right)_{nn}
	\\
	&= \sum\limits_{k = 1}^n {C_k^* f\left( {A_k } \right)C_k }  -
	\sum\limits_{k = 1}^n {C_k^* f\left( {\left| {A_k   - \sum\limits_{j = 1}^n {C_j^* A_jC_j  }   }
			\right|} \right)C_k }.
	\end{align*}
	It remains to mention that,   when  the column is just unital, we
	extend it to the unitary  $\left(n+1\right)$-column $\left(C_1,
	\cdots,C_n,0\right)$ and choose $A_{n+1}$ arbitrarily, but with
	spectrum in $I$, (see \cite{AF}). By the first part of the proof we therefore have
	\begin{align*}
	f\left( {\sum\limits_{k = 1}^n {C_k^* A_k C_k } } \right) &= f\left( {\sum\limits_{k = 1}^{n + 1} {C_k^* A_k C_k } } \right) \\
	&\le \sum\limits_{k = 1}^{n + 1} {C_k^* f\left( {A_k } \right)C_k }  - \sum\limits_{k = 1}^{n + 1} {C_k^* f\left( {\left| {A_k  -  \sum\limits_{j = 1}^{n + 1} {C_j^*A_jC_j }   } \right|} \right)C_k }  \\
	&= \sum\limits_{k = 1}^n {C_k^* f\left( {A_k } \right)C_k }  -
	\sum\limits_{k = 1}^n {C_k^* f\left( {\left| {A_k  -  \sum\limits_{j = 1}^n {C_j^*A_j C_j}  }
			\right|} \right)C_k }
	\end{align*}
	and thus the proof of the statement (2) is completely established.
	
	$(2)\Rightarrow (3)$.  Hold.
	
	$(3)\Rightarrow (1)$. Let $A$ and $B$ be positive and bounded
	linear operators with spectra in $I$ and $0\le \lambda\le
	1$.
	
	Consider
	\begin{align*}
	X = \left( {\begin{array}{*{20}c}
		A & 0  \\
		0 & B  \\
		\end{array}} \right),
	P = \left( {\begin{array}{*{20}c}
		{1_H } & 0  \\
		0 & 0  \\
		\end{array}} \right), Q= 1_{H \otimes H}  - P,
	\end{align*}
	\begin{align*}
	C = \left( {\begin{array}{*{20}c}
		{\sqrt \lambda  } & { - \sqrt {1 - \lambda } }  \\
		{\sqrt {1 - \lambda } } & {\sqrt \lambda  }  \\
		\end{array}} \right), \qquad\text{and}\qquad D  =
	\left( {\begin{array}{*{20}c}
		{\sqrt {1 - \lambda } } & { - \sqrt \lambda  }  \\
		{\sqrt \lambda  } & {\sqrt {1 - \lambda } }  \\
		\end{array}} \right).
	\end{align*}
	Then $C$ and $D$ are unitary operators on $H\oplus H$. We have
	\begin{align*}
	C^* XC = \left( {\begin{array}{*{20}c}{\lambda A + \left( {1 - \lambda } \right)B} & 0  \\
		0 & {\left( {1 - \lambda } \right)A + \lambda B}  \\
		\end{array}} \right),
	\end{align*}
	\begin{align*}
	D^* XD = \left( {\begin{array}{*{20}c}{\left( {1 - \lambda } \right)A + \lambda B} & 0  \\
		0 & {\lambda A + \left( {1 - \lambda } \right)B}  \\
		\end{array}} \right),
	\end{align*}
	\begin{align*}
	PC^* XCP = \left( {\begin{array}{*{20}c}
		{\lambda A + \left( {1 - \lambda } \right)B} & 0  \\
		0 & 0  \\
		\end{array}} \right)\qquad\text{and} \qquad QD^* XDQ = \left( {\begin{array}{*{20}c}
		0 & 0  \\
		0 & {\left( {1 - \lambda } \right)A + \lambda B}  \\
		\end{array}} \right).
	\end{align*}
	
	Thus, we have
	\begin{align*}
	&f\left( {\begin{array}{*{20}c}
		{\lambda A + \left( {1 - \lambda } \right)B} & 0  \\
		0 & {\left( {1 - \lambda } \right)A + \lambda B}  \\
		\end{array}} \right) \\
	&=\left( {\begin{array}{*{20}c}
		{f\left( {\lambda A + \left( {1 - \lambda } \right)B} \right)} & 0  \\
		0 & {f\left( {\left( {1 - \lambda } \right)A + \lambda B} \right)}  \\
		\end{array}} \right)\\
	&= f\left( {PC^* XCP + QD^* XDQ} \right) \\
	&\le Pf\left( {C^* XC} \right)P - Pf\left( {\left| C^* XC - PC^* XCP -  QD^* XDQ  \right|  } \right)P \\
	&\qquad+ Qf\left( {D^* XD} \right)Q-  Qf\left( { \left|D^* XD - PC^* XCP -  QD^* XDQ   \right| } \right)Q\qquad\qquad\qquad \text{(by \eqref{eq3.4})}\\
	&= PC^* f\left( X \right)CP - Pf\left( {\left| C^* XC - PC^* XCP -  QD^* XDQ  \right|  } \right)P \\
	&\qquad+ QD^* f\left( X \right)DQ-  Qf\left( { \left|D^* XD - PC^*
		XCP -  QD^* XDQ   \right| } \right)Q
	\\
	&= \left( {\begin{array}{*{20}c}
		{\lambda f\left( A \right) + \left( {1 - \lambda } \right)f\left( B \right)} & 0  \\
		0 & {\left( {1 - \lambda } \right)f\left( A \right) + \lambda f\left( B \right)}  \\
		\end{array}} \right) \\
	&\qquad- \left( {\begin{array}{*{20}c}
		{\lambda f\left( {\left( {1 - \lambda } \right)\left| {A - B} \right|} \right) + \left( {1 - \lambda } \right)f\left( {\lambda \left| {A - B} \right|} \right)} & 0  \\
		0 & {\left( {1 - \lambda } \right)f\left( {\lambda \left| {A - B} \right|} \right) + \lambda f\left( {\left( {1 - \lambda } \right)\left| {A - B} \right|} \right)}  \\
		\end{array}} \right).
	\end{align*}
	Hence, $f$ is operator superquadratic  on $I$ by seeing the
	$(1,1)$-components.
	
\end{proof}

\begin{remark}
	An operator convex version of Theorem  \ref{thm4} were proved by Hansen \& Pedersen in \cite{HP}. Therefore, in case $f$ is positive the inequality \eqref{eq3.3} could be considered as a new refinement of that result proved in \cite{HP}.
\end{remark}
A refinement of the classical Jensen's inequality \eqref{eq1.9} could be elaborated as follows:
\begin{corollary}
	\label{cor3.1}    Let $f:I\to \mathbb{R}$ be a real-valued continuous function. Let
	$A$ be a positive operator of a Hilbert space $\mathscr{H}$ with
	spectra contained in $I$. If $f$ is an operator superquadratic function, then
	the inequality
	\begin{align}
	f\left( {C^* A C} \right) \le C^* f\left( {A} \right)C -C ^* f\left( {\left| {A -C^* AC   }\right|} \right)C  \label{eq3.5}
	\end{align}
	holds for every operator $C$ on $\mathscr{H}$ that satisfy the condition
	$C^* C  = 1$.  
\end{corollary}
\begin{proof}
	Follows from Theorem \ref{thm4} by setting $n=1$.
\end{proof}
\begin{remark}
	Let $f:I\to \mathbb{R}$ be a real-valued continuous function. Let
	$A$ be a positive operator of a Hilbert space $\mathscr{H}$ with
	spectra contained in $I$. If $f$ is an operator subquadratic function, then the inequality 
	\begin{align*}
	f\left( {C^* A C} \right) \ge C^* f\left( {A} \right)C -C ^* f\left( {\left| {A -C^* AC   }\right|} \right)C  
	\end{align*}
	holds for every operator $C$ on $\mathscr{H}$ that satisfy the condition
	$C^* C  = 1$.  Furthermore, by applying the subquadratic function $f(t)=t^r$, $t>0$ $(r\in [0,1])$, then we have
	\begin{align*}
	\left( {C^* A C} \right)^r \ge C^* A^rC -C ^*  \left| {A -C^* AC   }\right|^rC  
	\end{align*}
	for all $r\in \left[0,1\right]$.
\end{remark}
A generalization of \eqref{eq3.5} (also, \eqref{eq1.7} and \eqref{eq1.8}) for any positive unital linear
map  between two Hilbert spaces having the same dimension is embodied in the following result.
\begin{theorem}
	Let  $\mathscr{H},\mathscr{K}$   be two Hilbert spaces such that ${\rm{dim}}\mathscr{H}= {\rm{dim}}\mathscr{K}$.  Let $f:J\to \mathbb{R}$ be a real-valued continuous function. Let
	$A$ be a positive  of a Hilbert space $\mathscr{H}$ with
	spectra contained in $J$,
	and consider  $\Phi:\mathscr{B}\left(\mathscr{H} \right)\to \mathscr{B}\left(\mathscr{K} \right)$ be a positive unital linear map. If $f$ is operator superquadratic function, then 
	the inequality
	\begin{align}
	f\left( {\Phi \left( A \right)} \right) \le \Phi \left( {f\left( A \right)} \right) - \Phi \left( {f\left( {\left| {A - \Phi \left( A \right)} \right|} \right)} \right)  \label{eq3.6}
	\end{align}
	holds. If  $f$ is operator subquadratic, then the inequality \eqref{eq3.6} is reversed. Thus, the following refinement of \eqref{eq1.7}
	\begin{align*}
	\Phi^2 \left( A \right)  \le \Phi \left( {A^2} \right) - \Phi \left( { \left| {A - \Phi \left( A \right)} \right|^2} \right) 
	\end{align*}
	is valid.
\end{theorem}
\begin{proof}
	Let $A \in \mathscr{B}\left(\mathscr{H}\right)$ be positive. Assume that $\mathcal{A}$ is the $C^*$-subalgebra of
	$\mathscr{B}\left(\mathscr{H}\right)$ generated by $A$ and $1_{\mathscr{H}}$. Without loss of generality, we may assume that
	$\Phi$ is defined on $\mathcal{A}$. Since every unital positive linear map on a commutative $C^*$-algebra
	is completely positive. It follows   that $\Phi$ is completely positive.
	So there exists  (by Stinespring's theorem \cite{S}), some isometry $V:\mathscr{H}\to \mathscr{K}$; and a unital $*$-homomorphism $\rho$ from $\mathcal{A}$ into the $C^*$-algebra $\mathscr{B}\left(\mathscr{H}\right)$ such that
	$\Phi(A) = V^*\rho(A)V$. Clearly, $f(\rho(A)) = \rho(f(A))$, for all continuous function $f$. Thus,
	\begin{align*}
	f\left( {\Phi \left( A \right)} \right) &= f\left( {P\rho \left( A \right)P} \right) \\
	&\le    P f\left( {\rho \left( A \right)} \right)P - P f\left( {\left| {\rho \left( {A - P\rho \left( A \right)P } \right)} \right|} \right)P \qquad \text{(by\, \eqref{eq3.4} \,with \,}n=1)\\
	&= P \rho \left( {f\left( A \right)} \right)P - P \rho \left( {f\left( {\left| {A -\Phi \left( A \right) } \right|} \right)} \right)P \\
	&= \Phi \left( {f\left( A \right)} \right) - \Phi \left( {f\left( {\left| {A -\Phi \left( A \right) } \right|} \right)} \right).
	\end{align*}
	which proves the required inequality. The last inequality holds by applying \eqref{eq3.6} to the superquaratic function $f(t)=t^2$, $\forall t>0$.

\end{proof}

The inequality \eqref{eq3.6} can be embodied in multiple versions
as stated in the following result.
\begin{corollary}
	\label{cor3}  Let    $\mathscr{H},\mathscr{K}$   be two Hilbert spaces such that ${\rm{dim}}\mathscr{H}= {\rm{dim}}\mathscr{K}$.  Let $f:\left[0,\infty\right) \to \mathbb{R}$ be a
	real-valued continuous function.
	and consider  $\Phi_k:\mathscr{B}\left(\mathscr{H} \right)\to \mathscr{B}\left(\mathscr{K} \right)$  $(k=1,\cdots,n)$ be a positive linear mappings with $\sum\limits_{k = 1}^n {\Phi _k \left( 1_{\mathscr{H}} \right)}  = 1_{\mathscr{K}}$. Then,  $f$ is operator superquadratic function if and only if
	\begin{align*}
	f\left( {\sum\limits_{k = 1}^n {\Phi _k \left( {A_k } \right)} } \right) \le \sum\limits_{k = 1}^n {\Phi _k \left( {f\left( {A_k } \right)} \right)}  - \sum\limits_{k = 1}^n {\Phi _k \left( {f\left( {\left| {A_k  -  \sum\limits_{j = 1}^n {\Phi _j \left( {A_j } \right) }   } \right|} \right)} \right)}
	\end{align*}
	for all positive operators $A_1\cdots,A_n$ in $ \mathscr{B}\left(\mathscr{H}\right)$.
\end{corollary}
\begin{proof}
	The proof is obvious.
\end{proof}



\end{document}